\documentclass[11pt]{amsart}
\usepackage{amssymb,amsmath,mathtools}
\usepackage{enumerate}
\usepackage{xcolor}
\usepackage{caption}
\usepackage{graphicx}
\usepackage{color}
\usepackage{amsmath}
\usepackage{amssymb}
\usepackage{amscd}
\usepackage{framed}
\usepackage{bbm}

\usepackage{fancyhdr,a4wide}
\usepackage{amsthm}

\newcommand{\R}{\mathbb{R}}
\newcommand{\inr}[1]{\left\langle #1 \right\rangle}

\newcommand{\E}{\mathbb{E}}

\newtheorem{theorem}{Theorem}[section]
\newtheorem{proposition}[theorem]{Proposition}
\newtheorem{lemma}[theorem]{Lemma}
\newtheorem{corollary}[theorem]{Corollary}
\theoremstyle{definition}
\newtheorem{definition}[theorem]{Definition}
\newtheorem{assumption}[theorem]{Assumption}

\newtheorem{remark}[theorem]{Remark}

\newtheorem{question}[theorem]{Question}

\renewcommand{\P}{\mathbb{P}}

\newcommand{\proj}{\mathrm{proj}}

\newcommand{\OD}{\mathcal{O}}
\newcommand{\ODN}{\mathcal{M}}

\newcommand{\cdim}{k_\ast(T)}

\renewcommand{\tilde}{\widetilde}

\numberwithin{equation}{section}

\newcommand{\eps}{\varepsilon}

\begin{document}

\title{Random embeddings with an almost Gaussian distortion}

\author{Daniel Bartl}
\address{Department of Mathematics, University of Vienna, Austria}
\email{daniel.bartl@univie.ac.at}
\author{Shahar Mendelson}
\address{Department of Mathematics, University of Vienna, Austria}
\email{shahar.mendelson@gmail.com}
\date{\today}

\begin{abstract}
Let $X$ be a symmetric, isotropic random vector in $\mathbb{R}^m$ and let $X_1...,X_n$ be independent copies of $X$. We show that under mild assumptions on $\|X\|_2$ (a suitable thin-shell bound) and on the tail-decay of the marginals $\langle X,u\rangle$, the random matrix $A$, whose columns are $X_i/\sqrt{m}$ exhibits a Gaussian-like behaviour in the following sense: for an arbitrary subset of $T\subset \mathbb{R}^n$, the distortion $\sup_{t \in T} | \|At\|_2^2 - \|t\|_2^2 |$ is almost the same as if $A$ were a Gaussian matrix.

A simple outcome of our result is that if $X$ is a symmetric, isotropic, log-concave random vector and $n \leq m \leq c_1(\alpha)n^\alpha$ for some $\alpha>1$, then with high probability, the extremal singular values of $A$ satisfy the optimal estimate: $1-c_2(\alpha) \sqrt{n/m} \leq \lambda_{\rm min} \leq \lambda_{\rm max} \leq 1+c_2(\alpha) \sqrt{n/m}$.
\end{abstract}

\maketitle
\setcounter{equation}{0}
\setcounter{tocdepth}{2}

\section{Introduction} \label{sec:intro}
Random embeddings have been studied extensively over the last few decades as a way of exposing a set's structure. Intuitively, a set $T \subset \R^n$ can be embedded in $\R^m$ without distorting most of its metric structure (e.g., almost preserving mutual distances between most of its points) as long as $m$ is at least as large as $T$'s effective dimension. The meaning of `effective dimension' is rather vague at this point, and various features of the set could serve as its effective dimension; our choice will be clarified in what follows.

The Johnson-Lindenstrauss Lemma \cite{MR737400} was one of the first examples to exhibit this general phenomenon. It states that a subset $T$ of a Hilbert space, that is of cardinality $n$, can be embedded in $\ell_2^m$ for $m=c(\eps)\log n$, with the embedding distorting mutual distances between all pairs of points in $T$ by at most a multiplicative factor of $1 \pm \eps$. Moreover, the embedding is random---in the original proof of the Johnson-Lindenstrauss Lemma the embedding was a normalized orthogonal projection onto a subspace selected according to the Haar measure on a Grassmann manifold of the right dimension. That choice of a random embedding is by no means unique, and over the years many different random ensembles have been shown to be, with high probability, embeddings that satisfy the Johnson-Lindenstrauss Lemma.

However, the dimension $m =c(\eps) \log n$ is, at times, loose. It is a rather coarse upper estimate on any reasonable notion of effective dimension: it does not distinguish between two sets that have the same cardinality but might have totally different metric structures.

The body of work that studies various notions of random embeddings is substantial and we make no attempt to survey it here. 
Rather, we focus on one particular way of embedding a subset of $\R^n$: a random ensemble $A=\frac{1}{\sqrt{m}}\tilde{A}$, where $\tilde{A}:\R^n \to \R^m$ is a random matrix whose columns are independent copies of a symmetric random vector $X$ in $\R^m$. 
We assume throughout that $X$ is symmetric and isotropic, i.e., that $X$ and $-X$ have the same distribution and that its covariance is the identity (implying, in particular, that for every $u \in \R^m$, $\E \inr{X,u}^2=\|u\|_2^2$). 
Thus, for every $t \in \R^n$,
$\E \|At\|_2^2 = \|t\|_2^2$, which explains the normalization of $1/\sqrt{m}$ used in the definition of $A$.

Our goal is to study when a satisfactory estimate on
\begin{equation} \label{eq:uniform-conc}
\sup_{t \in T} \left| \|At\|_2^2-\|t\|_2^2 \right|
\end{equation}
is possible for an arbitrary $T \subset \R^n$. While an upper estimate on \eqref{eq:uniform-conc} seems a rather weak notion of structure preservation, when applied to $T=V-V = \{ x - y : x,y \in V\}$ for $V \subset \R^n$, \eqref{eq:uniform-conc} captures the worst distortion of distances in $V$ caused by the embedding $A$.

It is important to stress the significant difference between the case we focus on here---when $A$ has $n$ independent copies of a random vector in $\R^m$ as columns, and the more standard scenario of a random embedding, in which the random ensemble has $m$ independent copies of a random vector in $\R^n$ as rows. The fact that the columns of $A$ are independent, combined with our goal of a nontrivial estimate for any $T \subset \R^n$, automatically leads to an obstruction that has to be overcome by an additional assumption on $X$. Indeed, one of the possible sets $T$ is very simple---the standard basis $\{e_1,...,e_n\}$. For such a set $T$ we have
\begin{equation} \label{eq:max-n-thin-shell}
\sup_{t \in T} \left| \|At\|_2^2-\|t\|_2^2 \right| = \max_{1 \leq i \leq n} \left| \|Ae_i\|_2^2 - 1 \right| =\max_{1 \leq i \leq n} \left| \frac{\|X_i\|_2^2}{m}-1 \right|,
\end{equation}
implying that at the very least,
\begin{equation} \label{eq:max-n-thin-shell-1}
\max_{1 \leq i \leq n} \left| \frac{\|X_i\|_2^2}{m}-1 \right|
\end{equation}
has to be well-behaved with reasonable probability. That can only happen if $\|X\|_2^2$ exhibits a sufficiently strong concentration around its mean---a so-called \emph{thin-shell bound}. Thin-shell bounds are known for a variety of random vectors (for estimates in the log-concave case, see, e.g. \cite{MR2285748,MR2846382,MR4244847} and Corollary \ref{cor:log-concave}, below).

Because it is unreasonable to consider embeddings that fail when the given set contains the standard basis, it makes sense to focus only on ensembles for which \eqref{eq:max-n-thin-shell-1} is small. Thus, the key question is what happens when this obstruction is resolved:

\begin{question} \label{qu:main}
Assume that $X$ is well-behaved, in the sense that with probability at least $1-\gamma$,
$$
\max_{1 \leq i \leq n} \left| \frac{\|X_i\|_2^2}{m}-1 \right| \leq \delta.
$$
When is it possible to control \eqref{eq:uniform-conc} in a satisfactory way for every $T \subset \R^n$?
\end{question}

To find what `satisfactory' should mean here, let us turn to the obvious example, when $A$ is generated by independent copies of the standard Gaussian random vector $G$; that is, the random vector whose coordinates are independent, standard Gaussian random variables.

\begin{remark}
In what follows we will not specify the dimension of each Gaussian random vector we encounter; that will be clear from the context in each instance.
\end{remark}

The way a Gaussian random matrix acts on an arbitrary set $T$ is well understood. To describe the features that are relevant in the context of Question \ref{qu:main} we require two important notions. The first one is the \emph{Gaussian mean-width} associated with $T \subset \R^n$; that is,
$$
\ell_*(T)=\E \sup_{t \in T} |\inr{G,t}| = \E \sup_{t \in T} \left|\sum_{i=1}^n g_i t_i \right|.
$$
To explain the reasoning for its name, note that
\[ \ell_\ast(T)\sim \sqrt{n} \int_{S^{n-1}} \sup_{t\in T} |\inr{\theta,t}| \, d\sigma(\theta) \]
where $\sigma$ denotes the Haar measure on the Euclidean unit sphere $S^{n-1}$; in other words, $\ell_\ast(T)$ is equivalent to the normalized average over all directional widths $\sup_{t\in T} |\inr{\theta,t}|$ when assigning equal weight to all directions $\theta\in S^{n-1}$.

The second notion is the so-called \emph{Dvoretzky-Milman} dimension (or \emph{critical dimension}). Here and in what follows denote
$$
d_T=\sup_{t \in T} \|t\|_2,
$$
and set
$$
k_*(T)=\left(\frac{\ell_*(T)}{d_T}\right)^2
$$
to be the Dvoretzky-Milman dimension. For reasons that will be clarified immediately, it will serve as our choice of the set's effective dimension.
More information on these notions and their role in the study of \emph{Asymptotic Geometric Analysis} can be found in \cite{MR3331351}.

One of the main achievements of Asymptotic Geometric Analysis is Milman's version of Dvoretzky's Theorem \cite{MR0293374}. For the Gaussian formulation we use here see, e.g., \cite{MR1036275} or \cite{MR3331351}. 
\begin{theorem} \label{thm:DM}
Let $T$ be a convex body---that is, a convex, bounded, centrally symmetric subset of $\R^n$ with a nonempty interior. Set $k_*=k_*(T)$ and $\eps>0$, and let $m=c(\eps) k_*$ for a well chosen constant $c(\eps)$ that depends only on $\eps$.
Let $A:\R^n \to \R^m$ be the normalized Gaussian matrix.
Then, with probability at least $1-2\exp(-c_1 k_*)$,
$$
(1-\eps) \frac{\ell_*(T)}{\sqrt{m}} B_2^m \subset A T \subset (1+\eps) \frac{\ell_*(T)}{\sqrt{m}} B_2^m,
$$
where $B_2^m$ is the Euclidean unit ball in $\R^m$.
\end{theorem}

At the same time, the Gaussian random matrix $A$ satisfies a uniform concentration estimate that holds for every $T \subset \R^n$. Indeed, one can show that for $u>0$, with probability at least $1-2\exp(-c_0u^2 k_*)$,
\begin{equation} \label{eq:gaussian-func-Bern}
\sup_{t \in T} \left| \|At\|_2^2 - \|t\|_2^2 \right| \leq c_1 \left(u d_T \frac{\ell_*(T)}{\sqrt{m}} + u^2 \frac{\ell_*^2(T)}{m}\right).
\end{equation}
A more general version of \eqref{eq:gaussian-func-Bern} can be found in \cite{MR3565471}.

The performance of the Gaussian random operator will be our benchmark. Thus, at least at the intuitive level, one should expect that when $m \lesssim k_*(T)$ the random embedding $A$ erases all of $T$'s metric structure, as is the case for the Gaussian ensemble. Indeed, if the complexity of two convex bodies, $T_1$ and $T_2$, is roughly the same---in the sense that the two have the same Euclidean diameter and mean-width---, the typical images $A T_1$ and $A T_2$ are almost the same: by Theorem \ref{thm:DM} both are close to a Euclidean ball of essentially the same radius.

However, once the dimension of the image space increases beyond $k_*(T)$, \eqref{eq:gaussian-func-Bern} implies that the Gaussian ensemble preserves some of $T$'s structure. Therefore, it is natural to ask whether a version of \eqref{eq:gaussian-func-Bern} is true for a random mapping $A$ whose independent columns are distributed according to a general symmetric, isotropic random vector rather than as the Gaussian one, and the interesting case is when $m\geq \cdim$.

Our main result is that, if the random ensemble generated by $X$ can overcome the trivial obstruction and deal with $T=\{e_1,...,e_n\}$, only a mild additional assumption on $X$ suffices to ensure that an almost Gaussian estimate holds in \eqref{eq:uniform-conc} for an arbitrary $T \subset \R^n$.

A hint on the features that $X$ needs to satisfy can be found in standard properties of the Gaussian vector in $\R^m$: in addition to a thin-shell bound---which the Gaussian random vector satisfies with parameters $\delta =c(\beta) \sqrt{\frac{\log n}{m}}$ and $\gamma \sim n^{-\beta}$ (e.g., by Bernstein's inequality)---, the Gaussian vector satisfies $L_p$-$L_2$ norm equivalences in the following sense: there is an absolute constant $C$ such that for every $x \in \R^m$ and every $p \geq 1$,
$$
\|\inr{G,x}\|_{L_p} \leq C \sqrt{p} \|\inr{G,x}\|_{L_2}.
$$

It turns out that a combination of a thin-shell bound and a mild tail decay on linear functionals suffices to (almost) recover the Gaussian estimate on \eqref{eq:uniform-conc}. To be more accurate, we require the following:

\begin{definition} \label{def:property-of-X}
We say that the symmetric, isotropic random vector $X$ in $\R^m$ is \emph{suitable} with constants $\delta\in[0,1]$, $\gamma \in (0,1)$, $\alpha \in (0,2]$, $R$ and $L$  if the following holds:

\begin{enumerate}
\item $X$ satisfies that with probability at least $1-\gamma$,
$$
\max_{1 \leq i \leq n} \left| \frac{\|X_i\|_2^2}{m} - 1 \right| \leq \delta.
$$
\item
for every $2 \leq p \leq R\log n$,
$$
\|\inr{X,x}\|_{L_p} \leq L p^{1/\alpha} \|\inr{X,x}\|_{L_2} \ \ \text{for every } \ x \in \R^m.
$$
\end{enumerate}
\end{definition}

Part $(1)$ of Definition \ref{def:property-of-X} is the bare-minimum: the thin-shell bound that is needed to remove the trivial obstruction caused by $\{e_1,...,e_n\}$. Part $(2)$ is an $L_p$-$L_2$ norm equivalence with constant $L$, corresponding to a $\psi_\alpha$ behaviour for some $0<\alpha\leq 2$ (rather than the $\psi_2$ behaviour exhibited by the Gaussian random vector)---but only up to $p =R \log n$.
The larger $R$ is, the wider the range of $L_p$ norms for which the norm equivalence holds, and that implies that linear functionals $\inr{X,x}$ have faster tail decays.

With this notion set in place, let us formulate our main result. Recall that $A:\R^n \to \R^m$ is a random matrix whose columns are $n$ normalized independent copies of the symmetric, isotropic random vector $X \in \R^m$.

\begin{theorem} \label{thm:main-intro}
Let $\beta \geq 1$ and assume that $X$ is suitable in the sense of Definition \ref{def:property-of-X}, with constants $\delta$, $\gamma$, $\alpha$, $R=R(\beta)$ and $L$.
Consider $T \subset \R^n$ that satisfies $k_*(T) \geq \log n$. Then with probability at least $1-\gamma-2\exp(-c_0k_*(T))-n^{-\beta}$,
$$
\sup_{t \in T} \left| \|At\|_2^2-\|t\|_2^2 \right| 
\leq 2\delta d_T^2+c(L,\alpha,\beta) \left(d_T\frac{\ell_*(T)}{\sqrt{m}} + \frac{\ell_*^2(T)}{m}\right)\log^{2/\alpha}\left(\frac{en}{k_*(T)}\right).
$$
\end{theorem}

A version of Theorem \ref{thm:main-intro} holds when $k_*(T) < \log n$ (see Section \ref{rem:critical.dim.log.n} below), but taking into account that even for the Gaussian random vector, the thin-shell bound implies that $\delta$ is of the order of $\sqrt{\frac{\log n}{m}}$, Theorem \ref{thm:main-intro} is more interesting when $k_*(T) \geq \log n$. For such sets, and up to the term needed for dealing with the trivial obstruction, the error is (almost) the same as in the Gaussian case, only incurring an additional logarithmic factor:
$$
\sim \left(d_T\frac{\ell_*(T)}{\sqrt{m}} + \frac{\ell_*^2(T)}{m}\right)\log^{2/\alpha}\left(\frac{en}{k_*(T)}\right).
$$

Theorem \ref{thm:main-intro} is actually rather surprising, for two different reasons. First of all, the error is almost the same as for the Gaussian ensemble, which is unexpected; after all, linear functionals $\inr{X,x}$ can exhibit a significantly heavier tail behaviour than the Gaussian one. There is no apparent source of mixing that would lead to a Gaussian-like behaviour, and so the origin of the almost Gaussian behaviour is rather mysterious. Secondly, when $T$ is very large---with $k_*(T) \sim n$, it seems that there is not enough independence in $A$ to justify a Gaussian error. We give a concrete example of this insufficient independence in Section \ref{sec:ex-log-concave} and explain why the thin-shell bound is the source of the extra mixing that compensates for the insufficient independence.

\subsection{Example: log-concave random vectors} \label{sec:ex-log-concave}

Let $X$ be a symmetric, isotropic, log-concave vector in $\R^m$; that is, it has a density that is a log-concave function (for an extensive survey on log-concavity and its role in Asymptotic Geometric Analysis, see \cite{MR3185453,MR3331351}). Observe that $X$ satisfies the assumption of Theorem \ref{thm:main-intro}:

\begin{enumerate}
\item
 The thin-shell estimate needed for Theorem \ref{thm:main-intro} follows from a recent result due to Chen \cite{MR4244847}: setting $\theta_m=\sqrt{\log{m} \log\log{m}}$, Chen's estimate on the \emph{thin-shell parameter} implies that there is an absolute constant $c$ such that for $u \geq 1$, with probability at least $1-2\exp(-u)$,
$$
\left| \frac{\|X\|_2^2}{m}-1\right| \leq cu\frac{\theta_m}{\sqrt{m}}.
$$
Therefore, $X$ satisfies the wanted thin-shell bound with probability $1-\gamma$ for
$$
\delta=c  \frac{ \theta_m }{\sqrt{m}} \cdot \log\left(\frac{en}{\gamma}\right).
$$
\item By Borell's Lemma (see, e.g., \cite{MR3331351}), a log-concave random vector satisfies a $\psi_1$-$L_1$ norm equivalence; that is, for every $x \in \R^m$ and every $p \geq 1$, $\|\inr{X,x}\|_{L_p} \leq C p \|\inr{X,x}\|_{L_1}$.
\end{enumerate}

Theorem \ref{thm:main-intro} leads to the following:

\begin{corollary} \label{cor:log-concave}
Let $T \subset \R^n$ satisfy $k_*(T) \geq \log n$. Then with probability $1-\gamma-2\exp(-c_0k_*(T))-n^{-\beta}$,
\begin{align*}
\sup_{t \in T} \left| \|At\|_2^2-\|t\|_2^2 \right| 
\leq & cd_T^2  \frac{\theta_m}{\sqrt{m}} \cdot \log\left(\frac{en}{\gamma}\right)
\\
+ & c(\beta) \left(d_T\frac{\ell_*(T)}{\sqrt{m}} + \frac{\ell_*^2(T)}{m}\right)\log^{2}\left(\frac{en}{k_*(T)}\right).
\end{align*}
\end{corollary}

In particular, setting $T=S^{n-1}$, Corollary \ref{cor:log-concave} is an estimate on the extremal singular values of $A$: if
$$
\sup_{t \in S^{n-1}} \left| \|At\|_2^2-1 \right| \leq \eps,
$$
then on that event the extremal singular values of $A$ satisfy that
\begin{equation} \label{eq:extremal}
1-\eps \leq \lambda_{\min}(A) \leq \lambda_{\max}(A) \leq 1+\eps.
\end{equation}
It is possible to show that the smallest $\eps$ one can hope for when $m \geq n$---say, with probability $1/2$---, is $\eps \sim \sqrt{\frac{n}{m}}$ even if $A$ is a Gaussian matrix. As it happens, such an optimal dependence holds for a wide variety of random matrices that have ``enough randomness". Specifically, when $m \geq cn$ for $c \gg 1$, the random matrix consists of rows that are $m$ independent copies of a random vector $Z \in \R^n$,  $Z$ is well-behaved in the sense that
$$
\E \max_{1 \leq i \leq m} \|Z_i\|_2 \leq C\sqrt{n}
$$
and linear functionals exhibit an $L_p$-$L_2$ norm equivalence for some $p>4$ \cite{MR3872323}.

This extra independence---the fact that there are more independent random vectors (the $m$ rows of the random matrix) than the dimension $n$ of the underlying space, is crucial when trying to control the way the matrix acts of $S^{n-1}$ (see, for example, the proofs in \cite{MR2601042,mendelson2014singular,MR3872323}).

In the case that interests us the situation is different. 
Seemingly, there is insufficient independence at one's disposal: $A$ consists of $n$ independent columns and to ensure a satisfactory estimate on the extremal singular values, $AS^{n-1}$ must be contained in a thin shell around $S^{m-1}$. 
Despite the lack of independence, Theorem \ref{thm:main-intro} is enough to yield the optimal estimate: we have that $k_*(S^{n-1}) \sim n$, and when $X$ is symmetric, isotropic and log-concave, then $\theta_m \leq c_0(\zeta) m^\zeta$ for any $\zeta>0$. Therefore, $X$ is suitable for
$$
\delta =c_0(\zeta) m^{\zeta-1/2}\cdot \log \left(\frac{en}{\gamma}\right)
$$
with parameters $\alpha=1$, $L$ that is an absolute constant, and $R$ that can be arbitrarily large.
Corollary \ref{cor:log-concave} implies that, with probability at least $0.9$,
$$
\sup_{t \in S^{n-1}} \left| \|At\|_2^2-1 \right|
\leq 2 \delta + c\sqrt{\frac{n}{m}}
\leq c_1(\zeta) \sqrt{\frac{n}{m}}
$$
provided that
$$
n \leq m \leq c_2(\zeta) \left(\frac{n}{\log^2 n}\right)^{1/2\zeta}.
$$
Thus, as long as $m$ is polynomial in $n$ and $n$ is sufficiently large, a typical realization of the random matrix $A$ satisfies the optimal estimate on the extremal singular values.

The key to the optimal estimate is the thin-shell bound: it was shown in \cite{MR517198,MR751274} (see also \cite{AnttilaBallPerissinaki,MR1959791}) that if $\|X\|_2^2/m$ exhibits non-trivial concentration around its mean, then for a typical $\theta \in S^{m-1}$, $\inr{X,\theta}$ is not far from a Gaussian random variable. Unfortunately, the notion of distance in this type of estimate is far too weak to be of use for us, but still this indicates where the extra mixing comes from: a Gaussian-like behaviour of enough marginals of $X$ that implies ``internal cancellation".


\subsection{Highlights of the argument}
Observe that for every $t \in \R^n$,
$$
\|At\|_2^2 = \left\|\frac{1}{\sqrt{m}}\sum_{i=1}^n t_i X_i \right\|_2^2 = \sum_{i=1}^n t_i^2 \frac{\|X_i\|_2^2}{m} + \frac{1}{m}\sum_{i \not = j} t_it_j \inr{X_i,X_j}.
$$
The second term is mean-zero, and so it is reasonable to expect that
$$
\sup_{t \in T} \left| \|At\|_2^2 - \|t\|_2^2 \right|
$$
is equivalent to the maximum of
$$
\sup_{t \in T} \left| \sum_{i=1}^n \left(\frac{\|X_i\|_2^2}{m}-1\right)t_i^2 \right|,
$$
and
$$
\sup_{t \in T} \left|\frac{1}{m}\sum_{i \not = j} t_it_j \inr{X_i,X_j}\right|.
$$
Clearly,
$$
\sup_{t \in T} \left| \sum_{i=1}^n \left(\frac{\|X_i\|_2^2}{m}-1\right)t_i^2 \right| \leq d_T^2 \max_{1 \leq i \leq n} \left| \frac{\|X_i\|_2^2}{m}-1 \right|,
$$
and following the premise of Question \ref{qu:main},
$$
\sup_{t \in T} \left| \sum_{i=1}^n \left(\frac{\|X_i\|_2^2}{m}-1\right)t_i^2 \right| \leq d_T^2  \delta \ \ \text{with probability at least } 1-\gamma.
$$
Hence, to control \eqref{eq:uniform-conc} one has to derive an appropriate bound on the second term.

As it happens, some features of the proof from \cite{mendelson2014singular}---on the extremal singular values of random matrices with iid rows---, play a central role in the proof of Theorem \ref{thm:main-intro}. We use some ideas from the argument in \cite{mendelson2014singular} to show that if $X$ is suitable then $A$ acts on sparse vectors almost as if it were a Gaussian matrix (see Theorem \ref{thm:random.matrix.satisfies.assumption} for an accurate formulation). At the same time, the argument from \cite{mendelson2014singular} cannot be used to prove a version of Theorem \ref{thm:main-intro} even in the case $T=S^{n-1}$, let alone for a general set $T$; there is not enough independence for that proof to work.

The other component of the proof is based on the method developed in \cite{mendelson2021column}. We consider a deterministic matrix $A$ that is well-behaved on sparse vectors (in a somewhat different sense than in \cite{mendelson2021column}), and show that by randomizing the columns of $A$ using independent signs $\eps_1,...,\eps_n$, one obtains a matrix $AD_\eps=A{\rm diag}(\eps_1,...,\eps_n)$ that exhibits an almost Gaussian behaviour. More accurately, for every $T \subset \R^n$ with $k_*(T) \geq  \log n$, and with high probability with respect to $(\eps_i)_{i=1}^n$,
$$
\sup_{t \in T} \left| \|AD_\eps t\|_2^2-\|t\|_2^2 \right| \leq 2\delta d_T^2+c \left(d_T\frac{\ell_*(T)}{\sqrt{m}} + \frac{\ell_*^2(T)}{m}\right)\log^{2/\alpha}\left(\frac{en}{k_*(T)}\right).
$$

The combination of those two facts (presented in Section \ref{sec:column.randomization} and in  Section \ref{sec:structural.estimate}), immediately leads to the proof of Theorem \ref{thm:main-intro}. Indeed, if $A$ has independent columns that are symmetric random vectors then $A$ and $AD_\eps$ have the same distribution.

\subsection{Notation}
Throughout \emph{absolute constants} are denoted by $c,c_1,c_2,\dots$. Their value may change from line to line. If a constant $c$ depends on a parameter $\beta$ we write $c=c(\beta)$;  $a\lesssim b$ denotes that $a\leq cb $ for an absolute constant $c$; and $a\sim b$ implies that both $a\lesssim b$ and $b\lesssim a$.

If $Y$ and $Z$ are independent random variables or vectors and $f$ is a measurable function of both, $\E_Y f $ denotes the expectation only w.r.t.\ $Y$, i.e.\ $\E_Y f=\int f(y,Z)\,d\P_Y(y)$.

The term \emph{random sign} refers to a Rademacher random variable, that is, a symmetric random variable taking the values $+1$ and $-1$.
Similarly, we use the term \emph{selector} for a $\{0,1\}$-valued random variable with mean $1/2$ and for iid selectors $\eta=(\eta_i)_{i=1}^n$, we set
\[ I_\eta:=\{ i\in \{1,\dots,n\} : \eta_i = 1\}. \]

For $I\subset\{1,\dots,n\}$ let $\proj_I\colon\mathbb{R}^n\to \mathbb{R}^n$ be the orthogonal projection on ${\rm span}(e_i)_{i \in I}$. Thus, $\proj_I x=\sum_{i \in I} x_i e_i$. If $(y_i)_{i=1}^n \in \R^n$ and $I\subset\{1,\dots,n\}$, set $y_I:=(y_i)_{i\in I}$; we  identify $(y_i)_{i \in I}$ with a vector in $\R^n$.

The \emph{cardinality} of a (finite) set $A$ is denoted by $|A|$.

\begin{definition}[Spheres and sparse vectors]
	For $I\subset \{1,\dots,n\}$, let
	\[ \mathcal{S}_I:=\{ x\in\mathbb{R}^n : \|x\|_2=1 \text{ and } x_i=0 \text{ for } i\in I^c\} \subset S^{n-1}.\]
	Thus, $\mathcal{S}_I$ is the Euclidean \emph{sphere} supported on coordinates in $I$.
	For $\ell \geq 1$, set
	\[ \mathcal{S}_{I,\ell}:=\{ x\in \mathcal{S}_{I} : 	|\{ i\in I : x_i\neq 0\}| \leq \ell \} \]
	to be the subset of $\mathcal{S}_I$ consisting of \emph{$\ell$-sparse} vectors.
\end{definition}

\section{Randomizing columns}
\label{sec:column.randomization}

Here we show that if a deterministic operator from $\mathbb{R}^n$ to $\mathbb{R}^m$ is well-behaved on sparse vectors, then an external randomization by $n$ independent random signs results in an operator that, with high probability, acts on an arbitrary sets almost as if it were Gaussian.

More precisely, in this section we fix a deterministic linear operator
\[ A\colon \mathbb{R}^n\to\mathbb{R}^m.\]
Let $\varepsilon=(\varepsilon_i)_{i=1}^n$ be a vector consisting of iid random signs, and consider the random operator $A_\varepsilon$ obtained by multiplying all columns of $A$ with the random signs:
\[ A_\varepsilon:=A\cdot
	\begin{pmatrix}
    \varepsilon_{1} & & \\
    & \ddots & \\
    & & \varepsilon_n
  \end{pmatrix}
  \colon\mathbb{R}^n\to\mathbb{R}^m. \]
Let $\eta=(\eta_i)_{i=1}^n$ be a vector whose coordinates are iid  selectors (that are also independent of $(\eps_i)_{i=1}^n$), recall that $I_\eta=\{i : \eta_i=1\}$, and define
\begin{align}
	\label{def:alpha.I}
	\begin{split}
	\OD_{I,2^s}
	&:=\max_{x\in \mathcal{S}_{I,2^s}} \max_{y\in \mathcal{S}_{I^c,2^s}} \langle Ax,Ay \rangle ,\\
	\OD_{2^s}
	&:=\E_{\eta} \OD_{I_\eta,2^s}
	\end{split}
\end{align}
for $s\geq 0$.

Before we go any further, let us explain what we mean by `$A$ is well-behaved' on sparse vectors.

\begin{assumption}[Assumptions on $A$]
\label{ass:A.deterministic.sparse}
	There are $\delta\in[0,1]$, $\alpha\in(0,2]$ and $C_A$ such that the following holds.
	\begin{enumerate}
	\item
	We have that
	\[\max_{1 \leq i \leq n} \big| \|A e_i\|_2^2-1 \big|
	\leq  \delta .\]
	\item For every $1 \leq 2^s \leq n$,
	\[
	\OD_{2^s}
	\leq C_A \max\Big\{ \sqrt \frac{ 2^s }{ m }  ,\frac{ 2^s}{ m } \Big\}  \log^{2/\alpha}\Big( \frac{en}{2^s} \Big)  .\]
	\end{enumerate}
\end{assumption}

\begin{remark}
The random matrix $A=\frac{1}{\sqrt{ m }}(X_1,\cdots,X_n)$ satisfies that $Ae_i=\frac{1}{\sqrt{ m }} X_i$. Thus, Part (1) of Assumption \ref{ass:A.deterministic.sparse} is just a thin-shell bound. And, we will show in Theorem \ref{thm:random.matrix.satisfies.assumption} that if $X$ is suitable then the matrix $A$ satisfies Part (2) of Assumption \ref{ass:A.deterministic.sparse} as well (with a non-trivial probability).
\end{remark}

The following is the main result of this section.
	
\begin{theorem}[Column randomization]
\label{thm:column.randomization}
	Suppose that Assumption \ref{ass:A.deterministic.sparse} is satisfied.
	Then there is an absolute constant $c_1$ and a constant $c_2=c_2( C_A)$  such that the following holds.
	
	For every set $T\subset \mathbb{R}^n$ satisfying $\cdim  \geq \log n$, with probability at least $1-\exp(-c_1 \cdim )$,
	\begin{align*}
	\sup_{t\in T} \big| \|A_\varepsilon t\|_2^2 -\|t\|_2^2 \big|
	&\leq 2 \delta d_T^2 + c_2\Big(\frac{d_T \ell_\ast(T) }{ \sqrt{m} }  + \frac{\ell^2_\ast(T) }{ m } \Big) \log^{2/\alpha}\Big( \frac{en}{ \cdim }\Big) .
	\end{align*}
\end{theorem}

\begin{remark}
The situation when $\cdim<\log n$ is explained in Section \ref{rem:critical.dim.log.n}.
\end{remark}

Let us point out that if
\begin{align}
\label{eq:delta.bound}
\delta\leq c \sqrt \frac{\cdim}{m}
\end{align}
for an absolute constant $c$, then the distortion caused by $A_\varepsilon$ coincides (up to the logarithmic term) with the distortion caused by a Gaussian operator:

\begin{corollary}
\label{cor:column.randomization}
	In the setting of Theorem \ref{thm:column.randomization}, if $\delta$ satisfies \eqref{eq:delta.bound}, then with probability at least $1-\exp(-c_1 \cdim )$,
	\begin{align*}
	\sup_{t\in T} \big| \|A_\varepsilon t\|_2^2 -\|t\|_2^2 \big|
	&\leq c_2\Big( \frac{d_T \ell_\ast(T) }{ \sqrt{m} } + \frac{\ell^2_\ast(T) }{ m } \Big) \log^{2/\alpha}\Big( \frac{en}{\cdim }\Big) .
	\end{align*}
\end{corollary}

\subsection{Preliminaries on chaining}

The proof of Theorem \ref{thm:column.randomization} is based on a chaining argument. A comprehensive study on chaining method and their applications can be found in \cite{talagrand2014upper}. In what follows we will only state the few facts that are needed in the proof of Theorem \ref{thm:column.randomization}.

\begin{definition}
	A family $(T_s)_{s\geq 0}$ of subsets of $T$ is an \emph{admissible sequence} if $|T_0|=1$ and $|T_s|\leq 2^{2^s}$ for $s\geq 1$.
	For $t\in T$, let $\pi_s t$ be a point in $T_s$ that minimizes the Euclidean distance to $t$. Set $\Delta_0t:=\pi_0t$ and $\Delta_s:=\pi_{s+1}t-\pi_s t$ for $s\geq 1$.
\end{definition}

The $\gamma_2$ functional of $T \subset (\R^n, \| \ \|_2)$ is defined by
\[ \gamma_2(T,\|\cdot\|_2)
:=\inf_{(T_s)_{s\geq 0} }\, \sup_{t\in T} \sum_{s\geq 0} \sqrt{2^s} \|\Delta_s t\|_2,
\]
where the infimum is taken over all admissible sequences of $T$.
By Talagrand's \emph{majorizing measure theorem} (see the presentation in \cite{talagrand2014upper}), there are absolute constants $c_1$ and $c_2$ such that, for every set $T\subset\mathbb{R}^n$,
\[
c_1 \ell_\ast(T) \leq \gamma_2(T,\|\cdot\|_2) \leq c_2 \ell_\ast(T).
\]

From now on, fix an (almost) optimal admissible sequence $(T_s)_{s\geq 0}$ for  $\gamma_2(T,\|\cdot\|_2)$.
\begin{remark}
Observe that by the definition of the $\gamma_2$ functional and the majorizing measure theorem,
\begin{align} \label{eq:mm}
\sup_{t\in T} \sum_{s\geq s'} \|\Delta_s t\|_2
\lesssim \frac{ \ell_\ast(T) }{ \sqrt{2^{s'}} }
\end{align}
for every $s'\geq 0$.
\end{remark}

Fix $0\leq s_0\leq s_1$ such that
\begin{align}
\label{eq:def.s0.s1}
\begin{split}
2^{s_0} & = \cdim,\\
2^{s_1}&=\max\{ \cdim , m \},
\end{split}
\end{align}
and assume without loss of generality that $s_0$ and $s_1$ are integers. To explain the choice of $s_1$, note that $\max\Big\{\sqrt\frac{2^s}{m}, \frac{2^s}{m}\Big\}$, a term that appears in the upper bound on $\OD_{2^s}$ in Assumption \ref{ass:A.deterministic.sparse}, is attained by $\sqrt\frac{2^s}{m}$ for $2^s<m$ and by $\frac{2^s}{m}$ otherwise. In the nontrivial case---when $\cdim \leq m$, $s_1$ is when the transition between the two occurs.

Now define
\begin{align*}
\Phi
&:=\sup_{t\in T} \sum_{s\geq s_1} \|A_\varepsilon \Delta_s t\|_2, \\
\Psi^2
&:= \sup_{t\in T} \big| \|A_\varepsilon \pi_{s_1}t\|_2^2 - \|\pi_{s_1}t\|_2^2 \big| .
\end{align*}

The first step in the proof of Theorem \ref{thm:column.randomization} is to decompose the distortion caused by $A_\varepsilon$ into several pieces:

\begin{lemma}
\label{lem:chaining}
	We have that
	\begin{align*}
	\sup_{t\in T} \big| \|A_\varepsilon t\|_2^2 - \|t\|_2^2 \big|
	&\leq \Psi^2 + 2\Phi \sqrt{\Psi^2 + d_T^2} + \Phi^2 + \sup_{t\in T} \big| \|\pi_{s_1} t\|_2^2 - \|t\|_2^2 \big|
	\end{align*}
	and
	\begin{align}
	\label{eq:error.t.pit}
	 \sup_{t\in T}  \big| \|\pi_{s_1} t\|_2^2 - \|t\|_2^2  \big|
	\leq \sup_{t\in T}\Big( \sum_{s \geq s_1} \|\Delta_st\|_2\Big)^2 + 2 d_T\sup_{t\in T}\sum_{s\geq s_1} \|\Delta_s t\|_2.
	\end{align}
\end{lemma}
\begin{proof}
	The proof is straightforward: for every $t\in T$,
	\begin{align} \label{eq:error-decom-1}
	\|A_\varepsilon t\|_2^2
	&=\|A_\varepsilon (t-\pi_{s_1}t) + A_\varepsilon \pi_{s_1}t\|_2^2  \nonumber
\\
	&=\|A_\varepsilon (t-\pi_{s_1}t)\|_2^2 + 2\langle A_\varepsilon (t-\pi_{s_1}t), A_\varepsilon \pi_{s_1}t\rangle + \|A_\varepsilon \pi_{s_1}t \|_2^2.
	\end{align}
	Writing $t-\pi_{s_1}t$ as the telescopic sum $\sum_{s\geq s_1} \Delta_s t$, the first term in \eqref{eq:error-decom-1} is dominated by $\Phi^2$.
	Moreover,
	\begin{align*}
	\|A_\varepsilon \pi_{s_1}t \|_2^2
	&\leq \big| \|A_\varepsilon \pi_{s_1}t \|_2^2- \|\pi_{s_1}t\|_2^2 \big| + d_T^2 \\
	&\leq \Psi^2 + d_T^2
	\end{align*}
	and by the Cauchy-Schwartz inequality,
	\begin{align*}
	\big| \|A_\varepsilon t\|_2^2 - \|t\|_2^2 \big|
	\leq \Psi^2 + 2\Phi\sqrt{\Psi^2 + d_T^2 } + \big| \|\pi_{s_1}t \|_2^2-\|t\|_2^2 \big|.
	\end{align*}
	As $t\in T$ was arbitrary, this completes the proof of the first claim.
	
	Turning to the second claim, note that
	\begin{align*}
	\big| \|t\|_2^2 - \|\pi_{s_1}t\|_2^2 \big|
	&\leq \|t-\pi_{s_1}t\|_2^2 + 2|\langle t-\pi_{s_1}t,\pi_{s_1}t\rangle|\\
	&\leq \|t-\pi_{s_1}t\|_2^2 + 2 \|t-\pi_{s_1}t\|_2 \| \pi_{s_1}t\|_2 \\
	&\leq \Big(\sum_{s\geq s_1} \|\Delta_s t\|_2\Big)^2 + 2 \sum_{s\geq s_1} \|\Delta_s t\|_2 d_T,	
	\end{align*}
	as required.
\end{proof}

Applying the estimate from \eqref{eq:mm} to \eqref{eq:error.t.pit}, it is evident that
\begin{align}
\label{eq:pit.t.error}
\begin{split}
\sup_{t\in T} | \|\pi_{s_1} t\|_2^2 - \|t\|_2^2 |
&\lesssim   \frac{ \ell_\ast^2(T) }{ 2^{s_1} } + \frac{ d_T \ell_\ast(T) }{ \sqrt{ 2^{s_1}} }\\
&\lesssim \frac{\ell_\ast^2(T)}{m} + \frac{d_T \ell_\ast(T)}{\sqrt m},
\end{split}
\end{align}
where the last inequality is actually an equality if $m \geq \cdim$.
Thus, we may now turn our attention to estimating $\Psi$ and $\Phi$, a task that requires some preparation.

\subsection{A decoupling argument}
\label{sec:decoupling}

Observe that for every $u\in\mathbb{R}^n$,
\begin{align*}
\|A_\varepsilon u\|_2^2
&=\sum_{i,j=1}^n \langle Ae_i,A e_j\rangle \varepsilon_i\varepsilon_j u_iu_j \\
&=\sum_{i=1}^n \|Ae_i\|_2^2 u_i^2 +  \sum_{i,j=1, \, i\neq j}^n \langle Ae_i,A e_j\rangle \varepsilon_i\varepsilon_j u_iu_j.
\end{align*}
Throughout we refer to the first term (and other terms of a similar nature) as the \emph{diagonal term} and to the second one as the \emph{off-diagonal term}.
Thanks to Part (1) of Assumption \ref{ass:A.deterministic.sparse}, the estimate on the diagonal term is immediate:
\begin{align*}
\Big| \sum_{i=1}^n \|Ae_i\|_2^2 u_i^2 -\|u\|_2^2  \Big|
&=\Big| \sum_{i=1}^n (\|Ae_i\|_2^2 -1) u_i^2 \Big|
\leq  \delta \|u\|_2^2.
\end{align*}
Therefore, sufficient control on the off-diagonal term would lead to an estimate on the error between $\|A_\varepsilon u\|_2^2$ and $\|u\|_2^2$.

To that end, let us introduce the following notation.

\begin{definition}
	For $I\subset\{1,\dots,n\}$ and $u,v\in\mathbb{R}^n$, define
	\begin{align*}
	Z_u
	&:=\sum_{i,j=1, \,i\neq j}^n \langle Ae_i,A e_j\rangle \varepsilon_i\varepsilon_j u_iu_j,\\
	W_{I,u}
	&:= A^\ast A \Big(\sum_{i\in I} \varepsilon_i u_i e_i\Big), \ \ \ {\rm and} \\
	V_{I,u,v}
	&:= \sum_{i\in I} \varepsilon_i u_i \big( W_{I^c,v} \big)_i.
	\end{align*}
\end{definition}

In addition, let $\eta$ be a vector consisting of iid selectors (that are also independent of $(\eps_i)_{i=1}^n$) and recall that $I_\eta=\{ i  : \eta_i=1\}$. A standard decoupling argument shows that
\begin{align}
\label{eq:Z.equals.EV}
4\E_\eta V_{I_\eta,u,u}
&= \sum_{i,j=1}^n \E_\eta   4\eta_i(1-\eta_j)   \langle Ae_i,Ae_j\rangle \varepsilon_i \varepsilon_j u_i u_j
= Z_u
\end{align}
and in particular,
\[ \|A_\varepsilon u\|_2^2
=\sum_{i=1}^n \|Ae_i\|_2^2 u_i^2 +4\E_\eta  V_{I_\eta,u,u}\]
for every $u\in\mathbb{R}^n$.

\subsection{Preliminary estimates on Bernoulli processes}

As will become clear in what follows, it is important to find estimates on $V_{I,u,v}$ that hold with high probability. And, by Markov's inequality, for a random variable $Y$ and $p \geq 1$,
\[ \P\Big( |Y|\geq e \big(\E |Y|^p\big)^{ 1/p } \Big)
\leq\exp(-p).\]
Thus, it suffices to estimate $L_p$-norms of $Y$ for suitable choices of $p$. In the context of a chaining procedure, the natural choices are $p=2^s$ for $s\geq 0$.

The following estimate on the growth of the (iterated) Bernoulli process $V$ is a crucial component in the proof of Theorem \ref{thm:column.randomization}.
To formulate it, recall that
$$
\OD_{I,p}=\max_{x\in \mathcal{S}_{I,p}} \max_{y\in \mathcal{S}_{I^c,p}} \langle Ax,Ay \rangle
$$
and
$$
\OD_{p}=\E_{\eta} \OD_{I_\eta,p}.
$$

\begin{proposition}
\label{prop:V.norm.estimate}
	There is an absolute constant $c$ such that the following holds.
	Let $u,v\in\mathbb{R}^n$ and set $I\subset\{1,\dots,n\}$. Then for $p \geq \log n$,	
\[ \big( \E_{\varepsilon} |V_{I^c,u,v}|^p  \big)^{ 1/p }
	\leq c \|u\|_2\|v\|_2 \OD_{I,p}.\]
\end{proposition}

\begin{corollary}
\label{cor:V.probability.estimate}
	For every $a>0$ there is a constant $c=c(a)$, such that the following holds.
	For $u,v\in\mathbb{R}^n$ and $2^s\geq \log n$, with $\P_{\varepsilon}$-probability at least
	\[1-\exp(-a 2^s),\]
	we have that
	\begin{align*}
	|Z_u|
	&\leq c \|u\|_2^2 \, \OD_{2^s}, \ \ \ {\rm and} \\
	\E_\eta |V_{I_\eta,u,v}|+ \E_\eta |V_{I^c_\eta,u,v}|
	&\leq c \|u\|_2\|v\|_2 \, \OD_{2^s}.
	\end{align*}
\end{corollary}
\begin{proof}
	We shall only prove the estimate on $\E_\eta|V_{I_\eta,u,v}|$; the required bounds on $\E_\eta |V_{I^c_\eta,u,v}|$ and $|Z_u|$  follow an identical path and are omitted.

	Set $f(\varepsilon,\eta)=|V_{I_\eta,u,v}|$ (as the latter depends on both random vectors $\varepsilon$ and $\eta$). The aim is to show that with $\P_\varepsilon$-probability at least $1-\exp(-a 2^s)$,
$$
\E_\eta f(\varepsilon,\eta)\leq c(a)\|u\|_2\|v\|_2 \, \OD_{2^s}.
$$	
	Set $p=2^s$ and denote by $\|\cdot\|_{L_p(\varepsilon)}$ the $L_p$-norm taken w.r.t.\ $\varepsilon$. By Jensen's inequality and the independence of $\eta$ and $\varepsilon$ it is evident that
	\begin{align*}
	\| \E_\eta f(\varepsilon,\eta) \|_{L_p(\varepsilon)}
	&\leq \E_\eta  \| f(\varepsilon,\eta) \|_{L_p(\varepsilon)} \\
	&\leq c_1 \|u\|_2\|v\|_2 \OD_{p},
	\end{align*}
	where the last inequality follows from Proposition \ref{prop:V.norm.estimate} and since $\OD_p=\E_\eta \OD_{I_\eta,p}$.
	
	By Markov's inequality,
	\begin{align*}
	\P_{\varepsilon}\Big( \E_{\eta}  f(\varepsilon,\eta) \geq c_2 \|u\|_2\|v\|_2 \OD_{p}\Big)
	&\leq \frac{ \E_{\varepsilon} \big| \E_{\eta} f(\varepsilon,\eta)\big|^p  }{ (c_2 \|u\|_2\|v\|_2 \OD_{p} )^p }  \\
	& \leq \frac{ c_1^p }{c_2^p}
	\leq \exp(-ap)
	\end{align*}
	for $c_2 := c_1\exp(a)$, which completes the proof.	
\end{proof}

The proof of Proposition \ref{prop:V.norm.estimate} is based on the growth rate of moments of Bernoulli processes, established in \cite{mendelson2021column} (see Equation (2.15) there). We present the proof of that estimate for the sake of completeness.

\begin{lemma}
\label{lem:bernoulli.interpolation}
	There is an absolute constant $c$ such that the following holds.
	Let $a,b\in\mathbb{R}^n$, $p\geq 1$, and $I\subset\{1,\dots,n\}$.
	Consider a partition of $I$ to the disjoint union $\cup_\ell I_\ell$, where $I_1\subset I$ is the set of the $p$ largest coordinates of $(|a_i|)_{i \in I}$, $I_2\subset I\setminus I_1$ is the set of the $p$ largest remaining coordinates and so on. Then
\[ \Big( \E \Big| \sum_{i\in I} \varepsilon_i a_ib_i \Big|^p\Big)^{1/p}
\leq c \|a\|_2 \max_{\ell} \| \proj_{I_\ell} b\|_2.\]
\end{lemma}
\begin{proof}
	
It is standard to verify that there is an absolute constant $c_1$ such that
	\begin{align}
	\label{eq:bernoulli.interpolation}
	\Big( \E \Big| \sum_{i\in I} \varepsilon_i a_ib_i \Big|^p\Big)^{1/p}
	\leq \sum_{i\in I_1} |a_ib_i| + c_1 \sqrt{p}\Big( \sum_{i\in I\setminus I_1} a_i^2b_i^2\Big)^{1/2}.
	\end{align}
	Indeed, first use the triangle inequality for the $L_p$-norm to split the sum over $I$ into a sum over $I_1$ and $I\setminus I_1$.
	To get the first term, use the $L_\infty$ bound $|\sum_{i\in I_1} \varepsilon_i a_ib_i| \leq \sum_{i\in I_1} |a_ib_i|$.
	To get the second term, use that $\sum_{i\in I\setminus I_1} \varepsilon_i a_ib_i$ is subgaussian with parameter $(\sum_{i\in I\setminus I_1} a_i^2b_i^2)^{1/2}$ (by Hoeffding's inequality) and the growth of the $L_p$-norms of subgaussian random variables (see, e.g., \cite[Exercise 3.2.4]{talagrand2014upper}).
	
	The first term in \eqref{eq:bernoulli.interpolation} can be bounded by the Cauchy-Schwartz inequality
	\begin{align*}
	\sum_{i\in I_1} |a_ib_i|
	\leq \|\proj_{I_1} a\|_2 \|\proj_{I_1} b\|_2
	\leq  \|a\|_2 \|\proj_{I_1} b\|_2.
	\end{align*}

	Turning to the second term, let $\|\cdot\|_\infty$ be the norm in $\ell_\infty^n$.
	By the choice of the sets $I_\ell$ we have that
	\[ \|\proj_{I_{\ell+1}} a\|_\infty
	\leq \frac{ \| \proj_{I_\ell}a\|_2 }{ \sqrt p } \]
	for every $\ell\geq 1$.
	Therefore
	\begin{align*}
	\sum_{i\in I\setminus I_1} a_i^2b_i^2
	= \sum_{\ell\geq 1} \sum_{i\in I_{\ell+1} } a_i^2b_i^2
	&\leq \sum_{\ell\geq 1} \|\proj_{I_{\ell+1}} a\|_\infty^2 \|\proj_{I_{\ell+1}} b\|_2^2  \\
	&\leq \Big( \sum_{\ell\geq 1} \frac{ \| \proj_{I_\ell}a\|_2^2 }{ p } \Big)\max_{\ell\geq 1} \|\proj_{I_{\ell+1}} b\|_2^2  \\
	&= \|a\|_2^2 \max_{\ell\geq 1} \frac{ \|\proj_{I_{\ell+1}} b\|_2^2}{p}.
	\end{align*}
	That, combined with \eqref{eq:bernoulli.interpolation} completes the proof.	
\end{proof}

\begin{proof}[Proof of Proposition \ref{prop:V.norm.estimate}]
	We use a twofold application of Lemma \ref{lem:bernoulli.interpolation} and Fubini's Theorem.
	Throughout fix $u,v,I,p$ as in the proposition.
	
	{\it Step 1.}
	To deal with the randomness originating from $\varepsilon_{I^c}$,  we work conditionally on $\varepsilon_I$.
	Observe that there is a set $\mathcal{Y}\subset \mathcal{S}_{I^c}$ of cardinality at most $\exp(c_1 p)$ such that
	\[ \big( \E_{\varepsilon_{I^c}} |V_{I^c,u,v}|^p \big)^{1/p}
	\leq c_2 \|u\|_2 \max_{y\in \mathcal{Y}} \sum_{i\in I} \varepsilon_i v_i (A^\ast A y)_i.\]
	
	Indeed, denote by $I^c=\cup_\ell J_\ell$ the partition of $I^c$, where $J_1$ consists of the $p$ largest coordinates of $(|u_i|)_{i \in I^c}$, $J_2$ consists of the $p$ largest remaining coordinates and so on.
	There are at most $\frac{n}{p}+1$ such sets $J_\ell$ and by
	Lemma \ref{lem:bernoulli.interpolation}
	\begin{align*}
	\big( \E_{\varepsilon_{I^c}} |V_{I^c,u,v}|^p \big)^{1/p}
	&= \Big( \E_{\varepsilon_{I^c}} \Big|\sum_{j\in I^c}\varepsilon_j u_j (W_{I,v})_j \Big|^p\Big)^{1/p} \\
	&\leq c_3 \|u\|_2 \max_\ell \|\proj_{J_\ell} W_{I,v}\|_2.
	\end{align*}
	
	Using a standard volumetric estimate and successive approximation (see, e.g., \cite[Corollary 4.1.15]{MR3331351}) for every $\ell$, it follows that there are sets $\mathcal{S}'_{J_\ell}\subset \mathcal{S}_{J_\ell}$, each of cardinality
	\[ |\mathcal{S}'_{J_\ell}| \leq \exp(c_4 p)\]
	such that for every $x \in \R^m$,
$$
\max_{y\in \mathcal{S}_{J_\ell}} \langle x,y\rangle
	\leq 2 \max_{y\in \mathcal{S}'_{J_\ell}} \langle x,y\rangle.
$$
In particular,
	\[ \|\proj_{J_\ell} x\|_2
	=\max_{y\in \mathcal{S}_{J_\ell}} \langle x,y\rangle
	\leq 2 \max_{y\in \mathcal{S}'_{J_\ell}} \langle x,y\rangle\]
	for every $x\in\mathbb{R}^m$.
	When applied to $x=W_{I,v}$ we have that
	\begin{align*}
	\|\proj_{J_\ell} W_{I,v}\|_2
	&\leq 2 \max_{y\in \mathcal{S}'_{J_\ell}} \sum_{i\in I} \varepsilon_i v_i (A^\ast A y)_i
	\end{align*}
	for every $\ell$.
	Setting $\mathcal{Y}:=\cup_\ell \mathcal{S}'_{J_\ell}$, then $|\mathcal{Y}| \leq \exp(c_1 p)$, because there are at most $\frac{n}{p}+1\leq 2n$ $\ell$'s, and $p\geq \log n$ by assumption.
	Note that the sets $(J_\ell)_\ell$ depend only on $u$.
	Hence the set $\mathcal{Y}$ depends only on $u$ as well, and, in particular, it does not depend on $\varepsilon_{I}$.
	Thus
	\[\big(\E_{\varepsilon_{I^c}} |V_{I^c,u,v}|^p \big)^{1/p}
	\leq c_2 \|u\|_2 \max_{y\in \mathcal{Y}} \sum_{i\in I} \varepsilon_i v_i (A^\ast A y)_i, \]
	which completes the proof of the first step.
	
	{\it Step 2.}
	Let us turn to the randomness originating from $\varepsilon_{I}$.
	Denote by $I=\cup_\ell I_\ell$ the partition of $I$, where $I_1$ consists of the $p$ largest coordinates of $(|v_i|)_{i \in I}$, $I_2$ are the remaining $p$-largest coordinates and so on.
	Then, for fixed but arbitrary $y\in \mathcal{Y}$, it follows from Lemma \ref{lem:bernoulli.interpolation} that
	\begin{align*}
	\Big( \E_{\varepsilon_{I}}  \Big| \sum_{i\in I} \varepsilon_i v_i (A^\ast A y)_i \Big|^p \Big)^{1/p}
	\leq c_3 \|v\|_2 \max_\ell \| \proj_{I_\ell} A^\ast Ay\|_2 .
	\end{align*}
	Next, still for an arbitrary $y\in \mathcal{Y}$,
	\begin{align*}
	\max_\ell\| \proj_{I_\ell} A^\ast Ay \|_2
	&=\max_{x\in \cup_\ell \mathcal{S}_{I_\ell}} \langle x ,  A^\ast Ay\rangle \\
	&\leq \max_{x\in \cup_\ell \mathcal{S}_{I_\ell}}\max_{z\in\cup_{\ell} \mathcal{S}_{J_\ell}} \langle Ax,Az\rangle
	\leq \OD_{I,p}
	\end{align*}
	where the first inequality follows as $\mathcal{Y}\subset \cup_{\ell} \mathcal{S}_{J_\ell}$ and the last inequality follows from the definition of $\OD_{I,p}$ and by noting that 
$$
\cup_\ell \mathcal{S}_{I_\ell}\subset \mathcal{S}_{I,p} \ \ {\rm  and} \ \  \cup_{\ell} \mathcal{S}_{J_\ell}\subset \mathcal{S}_{I^c,p}.
$$
	
Combining the two observations,
	\begin{align*}
	\Big( \E_{\varepsilon_{I}}  \max_{y\in \mathcal{Y}} \Big| \sum_{i\in I} \varepsilon_i v_i (A^\ast A y)_i \Big|^p \Big)^{1/p}
	&\leq  \Big( \sum_{y\in \mathcal{Y} }\E_{\varepsilon_{I}}  \Big| \sum_{i\in I} \varepsilon_i v_i (A^\ast A y)_i \Big|^p   \Big)^{1/p} \\
	&\leq c_4 | \mathcal{Y}|^{1/p}  \|v\|_2\OD_{I,p} \\
	&\leq c_5 \|v\|_2\OD_{I,p},
	\end{align*}	
	where the last inequality holds because $|\mathcal{Y}| \leq \exp(c_1p)$.
	
	{\it Step 3.}
	The proof is completed via an application of Fubini's Theorem:
	\begin{align*}
	\E_{\eps} |V_{I^c,u,v}|^p
	= \E_{\eps_{I}} \E_{\varepsilon_{I^c}} |V_{I^c,u,v}|^p
\leq & \E_{\eps_I} c_2^p \|u\|_2^p \max_{y \in \mathcal{Y}} \Big| \sum_{i\in I} \varepsilon_i v_i (A^\ast A y)_i \Big|^p
\\
\leq & c_2^p \|u\|_2^p c_5^p \|v\|_2^p \OD_{I,p}^p.
\end{align*}
\qedhere
\end{proof}

\subsection{Estimating $\Phi$}

The goal here is to estimate
\[
\Phi=\sup_{t\in T} \sum_{s\geq s_1} \|A_\varepsilon\Delta_s t\|_2
\]
when Assumption \ref{ass:A.deterministic.sparse} is satisfied. Recall that
\[2^{s_1}=\max\{ 2^{s_0}, m \} \geq\log n \]
and that for $u \in \R^n$,
$$
Z_u =\sum_{i,j=1, \,i\neq j}^n \langle Ae_i,A e_j\rangle \varepsilon_i\varepsilon_j u_iu_j.
$$

\begin{proposition}
\label{prop:Phi.estimate}
	There is an absolute constant $c_1>0$ and a constant $c_2=c_2(C_A)$ such that, with probability at least $1-\exp(-c_1 2^{s_1} )$,
	\[ \Phi \leq c_2\frac{ \ell_\ast(T)  }{  \sqrt m }  \log^{1/\alpha}\Big( \frac{en}{ \cdim  }  \Big)    .   \]
\end{proposition}
\begin{proof}
	Following the decoupling argument and invoking Assumption \ref{ass:A.deterministic.sparse}, it is evident that for every $t\in T$ and every $s\geq 0$,
	\[
	\|A_\varepsilon\Delta_s t\|_2^2
	\leq 2 \|\Delta_st\|_2^2  + |Z_{\Delta_st}|.\]
	
	Summing this inequality for $s\geq s_1$, we have, by \eqref{eq:mm}, that
	\begin{align*}
	\Phi
	&\leq c_3 \frac{ \ell_\ast(T) }{ \sqrt{2^{s_1}} } + \sup_{t\in T} \sum_{s\geq s_1} \sqrt{ |Z_{\Delta_st}| } \\
	&=:D_1+D_2.
	\end{align*}
	The choice of $s_1$ implies that $D_1\leq c_3\ell_\ast(t)/\sqrt{m}$.
	
	As for $D_2$, fix  $t\in T$ and $s\geq s_1$.
	 Corollary \ref{cor:V.probability.estimate} states that, with probability at least $1-\exp(-4\cdot 2^{s})$, 	\begin{align}
	\label{eq:Z.Delta.s.single.t}
	|Z_{\Delta_s} t|\leq c_4 \OD_{2^s} \| \Delta_s t\|_2^2.
	\end{align}
	Also,
	\[|\{\Delta_s t : t\in T\}|
	\leq | T_{s+1}| | T_s| 	
	\leq 2^{2^{s+1}+2^{s}}
	= 2^{3\cdot 2^s},\]
	and by the union bound, with probability at least $1-\exp(-c_5 2^{s})$, \eqref{eq:Z.Delta.s.single.t}  holds uniformly for every $t\in T$.
	Another application of the union bound, this time for $s \geq s_1$, and a comparison to an appropriate geometric progression yields that with probability at least
	\[ 1-\sum_{s\geq s_1} \exp(-c_5 2^s)
	\geq 1-\exp(-c_6 2^{s_1}), \]

	\[ |Z_{\Delta_s} t|
	\leq c_4 \OD_{2^s} \| \Delta_s t\|_2^2
	\quad\text{ for every } t\in T \text{ and }s\geq s_1.\]

	In particular, on that event,
	\begin{align}
	\label{eq:phi.help.1}
	D_2
	&\leq  \sqrt{ c_4 }  \sup_{t\in T} \sum_{s\geq s_1}  \sqrt{\OD_{2^s}} \|\Delta_s t\|_2.
	\end{align}
	All that remains is to note that, by Part (2) of Assumption \ref{ass:A.deterministic.sparse}, for every $s\geq s_1$
	\[ \OD_{2^s}
	\leq C_A \frac{2^s}{m} \log^{2/\alpha}\Big( \frac{en}{ \cdim } \Big). \]

	Hence, by the definition of the $\gamma_2$ functional and the majorizing measure theorem, on the event on which \eqref{eq:phi.help.1} holds,
	\[ D_2
	\leq c_7 \sqrt{C_A} \frac{\ell_\ast(T)}{\sqrt{m}} \log^{1/\alpha}\Big( \frac{en}{ \cdim }  \Big).\]
\end{proof}

\subsection{Estimating $\Psi$}

Let us estimate
\[ \Psi^2 = \sup_{t\in T} \big| \|A_\varepsilon \pi_{s_1}t\|_2^2- \|\pi_{s_1}t\|_2^2 \big| \]
when $A$ satisfies Assumption \ref{ass:A.deterministic.sparse} and $2^{s_0}=\cdim\geq \log n$.

\begin{proposition}
\label{prop:Psi.estimate}
	There is an absolute constant $c_1>0$ and a constant $c_2=c_2(C_A)$ such that, with probability at least
	$1-\exp(- c_1 2^{s_0} )$,
	\[ \Psi^2 \leq \delta d_T^2 +  c_2 \Big( \frac{ d_T\ell_\ast(T)}{\sqrt m} +\frac{\ell_\ast^2(T)}{m}\Big)  \log^{2/\alpha}\Big( \frac{en}{ \cdim } \Big)  . \]
\end{proposition}

The proof of Proposition \ref{prop:Psi.estimate} is based on the following lemma, from \cite{mendelson2021column} (see Lemma 2.3 there). We present the proof for the sake of completeness.

\begin{lemma} \label{lem:chaining.s0.s1}
	For every $t\in T$ we have that
	\[ |Z_{\pi_{s_1}t}|
	\leq |Z_{\pi_{s_0}t}| + 4\sum_{s=s_0}^{s_1-1} \E_\eta\Big( | V_{I^c_\eta,\Delta_s t, \pi_{s+1}t} | + |V_{I_\eta,\Delta_s t,\pi_st} | \Big) . \]
\end{lemma}

\begin{remark}
In case that $s_0=s_1$ the sum is on the empty set and the claim trivially true.
\end{remark}

\begin{proof}[Proof of Lemma \ref{lem:chaining.s0.s1}]
Let $t\in T$ and $I\subset\{1,\dots,n\}$. Writing $\pi_{s+1}t=\Delta_st+\pi_st$, it follows from the bi-linearity of $(u,v)\mapsto V_{I,u,v}$ that 	for every $s\geq 0$,
	\begin{align}
	\label{eq:V.bilineraity}
	\begin{split}
	V_{I,\pi_{s+1} t,\pi_{s+1} t }
	&=V_{I,\pi_{s+1} t,\Delta_s t } + V_{I,\pi_{s+1}t, \pi_{s} t}\\
	&=V_{I,\pi_{s+1} t,\Delta_s t } + V_{I,\Delta_s t , \pi_{s} t} + V_{I,\pi_{s}t, \pi_{s} t}.
	\end{split}
	\end{align}
Recall that $Z_u=4\E_\eta V_{I_\eta,u,u}$ for every $u\in\mathbb{R}^n$ (see \eqref{eq:Z.equals.EV}) hence an iterative application of \eqref{eq:V.bilineraity} shows that
	\begin{align*}
	Z_{\pi_{s_1}t}
	&= 4\E_\eta V_{I_\eta,\pi_{s_1} t,\pi_{s_1} t } \\
	&=4 \sum_{s=s_0}^{s_1-1} \E_\eta \Big(   V_{I_\eta,\pi_{s+1} t,\Delta_s t } + V_{I_\eta,\Delta_s t , \pi_{s} t} \Big)  + Z_{\pi_{s_0}t}.
	\end{align*}
Note that by the definition of $V$ and the fact that $\E \eta_i =1/2$, we have $\E_\eta V_{I_\eta,u,v}=\E_\eta V_{I^c_\eta,v,u}$ for every $u,v\in\mathbb{R}^n$.
\end{proof}
	
\begin{proof}[Proof of Proposition \ref{prop:Psi.estimate}]
	By a decoupling argument combined  with Lemma \ref{lem:chaining.s0.s1},
	\begin{align*}
	\Psi^2
	&\leq \sup_{t\in T} \Big( \delta \|\pi_{s_1}t\|_2^2  + |Z_{\pi_{s_1}t}| \Big) \\
	&\leq \delta d_T^2 +\sup_{t\in T} |Z_{\pi_{s_0}t}|
	 + \sup_{t\in T} 4\sum_{s=s_0}^{s_1-1} \E_\eta\Big( | V_{I^c_\eta,\Delta_s t, \pi_{s+1}t} | + | V_{I_\eta,\Delta_s t,\pi_st} | \Big) \\
	&=:\delta d_T^2  + D_1 + D_2.
	\end{align*}
		
	To control $D_1$, one may proceed as in the proof of Proposition \ref{prop:Phi.estimate}:
	by Corollary \ref{cor:V.probability.estimate} and the union bound, with probability at least $1-\exp(-c_3 2^{s_0})$,
	\begin{align}
	\label{eq:prop.V.estimate.Z0}
	D_1
	\leq c_4 d_T^2 \OD_{2^{s_0}}.
	\end{align}
	From this point on, let us distinguish between two cases:
	
	{\it Case 1:} $s_0=s_1$.
	In this case $D_2=0$ and $\cdim \geq m$. By Part (2) of Assumption \ref{ass:A.deterministic.sparse},
	\[ \OD_{2^{s_0}}
	\leq C_A \frac{ \cdim }{m} \log^{2/\alpha}\Big( \frac{en}{  \cdim } \Big).\]
	Therefore, recalling that $\cdim=\ell_\ast^2(T) / d_T^2$, on the event on which \eqref{eq:prop.V.estimate.Z0} holds,
	\[  D_1 \leq c_4 C_A \frac{\ell_\ast^2(T)}{m} \log^{2/\alpha}\Big( \frac{en}{ \cdim } \Big), \]
	as required.
		
	{\it Case 2:} $s_0<s_1$.
	In this case $m > \cdim$, and in particular,
	\[ \OD_{2^{s_0}}
	\leq C_A \sqrt \frac{ \cdim }{m}  \log^{2/\alpha}\Big( \frac{en}{ \cdim } \Big) . \]
	Therefore, on the event on which \eqref{eq:prop.V.estimate.Z0} holds,
	\[	D_1
	\leq c_4 C_A \frac{ d_T\ell_\ast(T) }{\sqrt m}  \log^{2/\alpha}\Big( \frac{en}{ \cdim } \Big). \]
	
	Next, to estimate $D_2$ one may proceed as in the proof of Proposition \ref{prop:Phi.estimate}. Clearly
	\[ |\{(\Delta_s t, \pi_{s+1}t) : t\in T\} |
	\leq 2^{2^{s+2}}2^{2^{s+1}}
	\leq 2^{2^{s+3}},\]
	and invoking Corollary \ref{cor:V.probability.estimate} followed by the union bound, we have that with $\P_\varepsilon$-probability at least
	\begin{align*}
	 1-\sum_{s=s_0}^{s_1-1} \exp(-c_5 2^s)
	&\geq 1-\exp(-c_6 2^{s_0} ), \\
	D_2
	&\leq c_7  \sup_{t\in T}\sum_{s=s_0}^{s_1-1} \|\Delta_s t \|_2 d_T \OD_{2^s}  .
	\end{align*}
	Moreover, by Part (2) of Assumption \ref{ass:A.deterministic.sparse}, for every  $s_0\leq s<s_1$,
	\[ \OD_{2^s}
	\leq C_A  \sqrt \frac{ 2^s }{ m }  \log^{2/\alpha}\Big( \frac{en}{ \cdim } \Big).
	\]
	Therefore, recalling the definition of the $\gamma_2$ functional and the majorizing measure theorem,
	\begin{align*}
	 D_2
	&\leq c_8C_A \frac{ d_T\ell_\ast(T) }{\sqrt m }  \log^{2/\alpha}\Big( \frac{en}{ \cdim } \Big).
	\end{align*}
\end{proof}

\subsection{Proof of Theorem \ref{thm:column.randomization}}
Observe that 
$\sqrt{\Psi^2 + d_T^2} \leq \Psi + d_T$, and therefore,
$$
\Phi\sqrt{\Psi^2 + d_T^2} 
\leq \Phi \Psi + \Phi d_T
\leq \frac{\Psi^2 + \Phi^2}{2} + \Phi d_T.
$$
By Lemma \ref{lem:chaining},
	\begin{align*}
	\sup_{t\in T} \big| \|A_\varepsilon t\|_2^2 - \|t\|_2^2 \big|
	&\leq 2\Psi^2 + 2\Phi d_T + 2\Phi^2 + \sup_{t\in T} \big| \|\pi_{s_1} t\|_2^2 - \|t\|_2^2 \big|,
	\end{align*}
and the claim follows from the estimate on $\Phi$ (Proposition \ref{prop:Phi.estimate}); the estimate on $\Psi$ (Proposition \ref{prop:Psi.estimate}); and the uniform estimate on the error between $\pi_{s_1}t$ and $t$ in \eqref{eq:pit.t.error}.

\subsection{The case $\cdim<\log n$}
\label{rem:critical.dim.log.n}
	The only difference when the standing assumption, that $\cdim\geq \log n$, is not satisfied, occurs in (the proof of) Theorem \ref{thm:column.randomization}.
	When that happens, one starts the chaining processes at
	\[2^{s_0}=\lambda_\ast(T):=\max\{\cdim,\log n\}\]
	 instead of $2^{s_0}=\cdim$ (compare with \eqref{eq:def.s0.s1}).
	 The proof follows the same path with this choice of $s_0$ and for $2^{s_1}=\max\{2^{s_0},m\}$:

\begin{theorem}
\label{thm:column.randomization.cdim.less.than.log.n}
	Suppose that Assumption \ref{ass:A.deterministic.sparse} is satisfied.
	Then there is an absolute constant $c_1$ and a constant $c_2=c_2( C_A)$  such that the following holds.
	
	For every set $T\subset \mathbb{R}^n$, with probability at least $1-\exp(-c_1 \lambda_\ast(T) )$,
	\begin{align*}
	&\sup_{t\in T} \big| \|A_\varepsilon t\|_2^2 -\|t\|_2^2 \big|\\
	&\leq 2\delta d_T^2 + c_2\Big(\frac{d_T \ell_\ast(T) }{ \sqrt{m} }  +\frac{\ell^2_\ast(T) }{ m } + d_T^2 \max\Big\{ \sqrt \frac{\lambda_\ast(T)}{m}, \frac{\lambda_\ast(T)}{m} \Big\} \Big) \log^{2/\alpha}\Big( \frac{en}{ \lambda_\ast(T) }\Big) .
	\end{align*}
\end{theorem}	

Note that when $\lambda_\ast(T)=\cdim$ (i.e., when $\cdim\geq\log n$), then Theorem \ref{thm:column.randomization.cdim.less.than.log.n} and Theorem \ref{thm:column.randomization} coincide.

Because the proof of Theorem \ref{thm:column.randomization.cdim.less.than.log.n} is identical to that of Theorem \ref{thm:column.randomization}, we omit the details.

\section{Structural estimates on the matrix generated by $X$}
\label{sec:structural.estimate}

Here we show that, with a nontrivial probability, the random matrix
\[  A:= \frac{1}{\sqrt{m}}(X_1,\cdots,X_n)\colon\mathbb{R}^n\to\mathbb{R}^m \]
satisfies Assumption \ref{ass:A.deterministic.sparse}.

To simplify notation, set $\mathbb{X}:=(X_i)_{i=1}^n$ and recall, once again, the definitions of
\begin{align*}
\OD_{I,2^s}
&= \max_{x\in\mathcal{S}_{I,2^s}}\max_{y\in\mathcal{S}_{I^c,2^s}} \langle Ax,Ay\rangle \\
&= \max_{x\in\mathcal{S}_{I,2^s}}\max_{y\in\mathcal{S}_{I^c,2^s}} \Big\langle \sum_{i\in I}  x_i \frac{X_i}{\sqrt m} ,  \sum_{j\in I^c} y_j \frac{X_j}{\sqrt m} \Big\rangle,\\
\OD_{2^s}
&=\E_\eta  \OD_{I_\eta,2^s}.
\end{align*}
As always, $\eta$ is a vector consisting of iid selectors with mean $1/2$ that are independent of $\mathbb{X}$  and $I_\eta=\{i : \eta_i=1\}$.

Also, set
\begin{align*}
\ODN_{I,2^s}
&:=\max_{x\in\mathcal{S}_{I,2^s}} \Big\| \sum_{i\in I} x_i \frac{X_i }{\sqrt m} \Big\|_2 ,\\
\ODN_{2^s}
&:=\ODN_{\{1,\dots,n\},2^s}
=\max_{\{x \in S^{n-1} : |{\rm supp}(x)| \leq 2^s\}} \Big\|\sum_{i=1}^n x_i \frac{X_i}{\sqrt{m}} \Big\|_2
\end{align*}
and recall that $X$ suitable with constants $\delta$, $\gamma$, $\alpha$, $R$ and $L$ if
\[ \P_{\mathbb{X}}\Big( \max_{1 \leq i \leq n} \Big| \frac{ \|X_i \|_2^2}{ m } -1 \Big|
\leq \delta \Big) \geq 1-\gamma\]
and for every $x \in \R^m$,
\[  \| \langle X,x\rangle \|_{L_p}
\leq L p^\frac{1}{\alpha}   \| \langle X,x\rangle \|_{L_2}
\quad\text{for } 2\leq p\leq R \cdot \log e n. \]

\begin{theorem}
\label{thm:random.matrix.satisfies.assumption}
	Let $\beta \geq 1$ and assume that $X$ is suitable with constants $\delta$, $\gamma$, $\alpha$, $R=R(\beta)$ and $L$.	
	Then there is a constant $c_1(\alpha,\beta,L)$ such that, with probability at least $1- \gamma - n^{-\beta}$,	for every $1\leq 2^s\leq n$,
	\begin{align*}
	&\OD_{2^s}
	\leq c_1 \max\Big\{ \sqrt \frac{ 2^s   }{ m}  , \frac{ 2^s }{ m} \Big\} \log^{2/\alpha}\Big( \frac{en}{2^s}\Big) .
	\end{align*}
\end{theorem}

\subsection{Proof of Theorem \ref{thm:random.matrix.satisfies.assumption}}

Throughout this section we fix the exponent $\beta$ and assume that $X$ is suitable with an appropriate constant $R=R(\beta)$ for a value that is specified in what follows. 
Also, denote by $(y^\ast_i)_{i\geq 1}$ the monotone nonincreasing rearrangement of $(|y_i|)_{i \in I}$; in particular, $y^\ast_1=\max_{i\in I } |y_i|$ and so on. Hence, for a vector $(y_i)_{i\in I}$,
\begin{align}
\label{eq:sparse.ordering.norm}
\max_{x\in \mathcal{S}_{I,2^s}} \sum_{i\in I} x_i y_i
= \Big( \sum_{i=1}^{2^s} y_i^{\ast 2} \Big)^{ 1/2 }.
\end{align}

\begin{lemma}[Dimension reduction]
\label{lem:dim.reduction}
	There are absolute constants $c_1,c_2$ and, for every $1\leq 2^r\leq n$ and $I\subset\{1,\dots,n\}$, there are sets
	\[\mathcal{S}_{I,2^r}'\subset \mathcal{S}_{I,2^r} \text{ such that } | \mathcal{S}_{I,2^r}'| \leq \exp\Big(c_1 2^r\log \frac{en}{2^r} \Big)\]
	for which the following holds:
	setting
	\begin{align*}
	\mathcal{E}_{I,2^r}
	&:= \max_{y\in \mathcal{S}'_{I^c,2^{r}}}  \frac{1}{\sqrt m} \Big( \sum_{i=2^{r-1}+1}^{2^r} \Big\langle X_i, \sum_{j\in I^c} y_j \frac{X_j}{\sqrt m}  \Big\rangle^{\ast 2}\Big)^{1/2},\\
	\mathcal{F}_{I^c,2^r}&:=\max_{x\in \mathcal{S}'_{I,2^{r-1}}} \frac{1}{\sqrt m} \Big( \sum_{j=2^{r-1}+1}^{2^r} \Big\langle   \sum_{i\in I} x_i \frac{X_i}{\sqrt m} , X_j\Big\rangle^{\ast 2}\Big)^{1/2},
	\end{align*}
	we have that		
	\begin{align}
	\label{eq:alpha.dimension.reduction}
	\OD_{I,2^s}
	&\leq c_2\Big(   \OD_{I,2^0} +  \sum_{r=1}^s \big( \mathcal{E}_{I,2^r} + \mathcal{F}_{I^c,2^r} \big) \Big).
	\end{align}
\end{lemma}
	Let us stress that, in the definition of $\mathcal{E}_{I,2^r}$, the monotone nonincreasing rearrangement is only with respect to coordinates in $I$.
	That is, the sum appearing in $\mathcal{E}_{I,2^r}$ equals $\sum_{i=2^{r-1}+1}^{2^r} y_i^{\ast 2}$ where $(y^\ast_i)_{i\geq 1}$ is the monotone nonincreasing rearrangement of the vector $(|y_i|)_{i\in I}$ defined by $y_i:=\langle X_i, \sum_{j\in I^c} y_j \frac{X_j}{\sqrt m} \rangle$ for $i\in I$.
	Similarly,  the monotone nonincreasing rearrangement in $\mathcal{F}_{I^c,2^r}$ is only with respect to coordinates in $I^c$.
\begin{proof}
	Let $c_3\geq 2$ be an absolute constant to be chosen later and set
	\[ \delta_{2^r}:=\Big( \frac{2^r}{en} \Big)^{c_3} \]
	for $1\leq 2^r\leq n$.
	Observe that there are $\binom{n}{2^r}$ subsets of $\{1,...,n\}$ of cardinality $2^r$. Therefore, by a standard volumetric estimate and a successive approximation argument, there are sets $\mathcal{S}_{I,2^r}'\subset \mathcal{S}_{I,2^r}$ of cardinality as at most
	\[ \exp\Big( 2^r\log \frac{en}{2^r\delta_{2^r}} \Big)
	= \exp\Big( (1+c_3) 2^r\log \frac{en}{2^r} \Big) \]
	such that
	\[\max_{x\in \mathcal{S}_{I,2^r}} \langle x,y\rangle
	\leq \frac{1}{1-\delta_{2^r}} \max_{x\in \mathcal{S}_{I,2^r}'} \langle x,y\rangle\]
	for every $y\in\mathbb{R}^n$.
	The analogous statement holds for $\mathcal{S}_{I^c,2^r}$.
	
	By the definition of $\OD_{I,2^r}$ and  interchanging the order of two maxima,
	\begin{align*}
	\OD_{I,2^r}
	&= \max_{x\in \mathcal{S}_{I,2^r}} \max_{y\in \mathcal{S}_{I^c,2^r}}
\langle A^\ast Ax,y\rangle
\leq \max_{x \in \mathcal{S}_{I,2^r}} \frac{1}{1-\delta_{2^{r}}} \max_{y\in \mathcal{S}_{I^c,2^r}'}\langle A^\ast Ax,y\rangle
\\
	&=  \frac{1}{1-\delta_{2^{r}}} \max_{y\in \mathcal{S}_{I^c,2^r}'} \max_{x\in \mathcal{S}_{I,2^r}} \langle Ax,Ay\rangle \\
	&=	\frac{1}{1-\delta_{2^{r}}}   \max_{y\in \mathcal{S}'_{I^c,2^r}} \frac{1}{\sqrt m} \cdot  \max_{x\in\mathcal{S}_{I,2^r}} \sum_{i\in I} x_i \Big\langle  X_i,\sum_{j\in I^c} y_j \frac{ X_j }{\sqrt m} \Big\rangle.
	\end{align*}
	Applying \eqref{eq:sparse.ordering.norm} to the last term, it is evident that
	\begin{align*}
	\OD_{I,2^r}
	&\leq \frac{1}{1-\delta_{2^{r}}}   \max_{y\in \mathcal{S}'_{I^c,2^r}} \frac{1}{\sqrt m} \cdot \Big(\sum_{i=1}^{2^r} \Big\langle  X_i,\sum_{j\in I^c} y_j \frac{ X_j}{\sqrt m}\Big\rangle^{\ast 2}\Big)^{ 1/2 }.
	\end{align*}
Splitting the sum over $i=1,\dots,2^r$ into a sum over $\{1,...,2^{r-1}\}$ and  over $\{2^{r-1}+1,...,2^r\}$,
	\begin{align*}
	&\max_{y\in \mathcal{S}'_{I^c,2^r}} \frac{1}{\sqrt m} \cdot \Big(\sum_{i=1}^{2^r} \Big\langle  X_i,\sum_{j\in I^c} y_j \frac{ X_j}{\sqrt m}\Big\rangle^{\ast 2}\Big)^{ 1/2 } \\
	&\leq \max_{y\in \mathcal{S}'_{I^c,2^r}} \frac{1}{\sqrt m} \cdot \Big(\sum_{i=1}^{2^{r-1}} \Big\langle  X_i,\sum_{j\in I^c} y_j \frac{ X_j}{\sqrt m}\Big\rangle^{\ast 2}\Big)^{ 1/2 } + \mathcal{E}_{I,2^r}.
	\end{align*}
Using \eqref{eq:sparse.ordering.norm} once again,
	\begin{equation*}
	\OD_{I,2^r}
	\leq\frac{1}{1-\delta_{2^{r}}} \Big( \max_{y\in \mathcal{S}'_{I^c,2^r}} \max_{x\in\mathcal{S}_{I,2^{r-1}}} \Big\langle \sum_{i\in I} x_i \frac{X_i}{\sqrt m} ,\sum_{j\in I^c}y_j \frac{ X_j }{\sqrt m}\Big\rangle + \mathcal{E}_{I,2^r}\Big).
	\end{equation*}
	
	Applying the same arguments (with the roles of $x$ and $y$ reversed), we obtain
	\begin{align*}
	\max_{y\in \mathcal{S}_{I^c,2^r}} \max_{ x \in\mathcal{S}_{I,2^{r-1}}} \Big\langle \sum_{i\in I} x_i \frac{X_i}{\sqrt m } , \sum_{j\in I^c} y_j \frac{ X_j }{\sqrt m }\Big\rangle
	&\leq \frac{1}{1-\delta_{2^{r-1}}}\Big( \OD_{I,2^{r-1}} + \mathcal{F}_{I^c,2^r} \Big)
	\end{align*}
	and combining the two estimates,
	\begin{align}
	\label{eq:dim.reduction.proof}
	\OD_{I,2^r}
	\leq  \frac{1}{1-\delta_{2^{r}}} \frac{1}{1-\delta_{2^{r-1}}} \Big( \OD_{I,2^{r-1}} + \mathcal{E}_{I,2^r} + \mathcal{F}_{I^c,2^r} \Big).
	\end{align}
	
	Finally note that
	\[ \prod_{1\leq 2^r\leq n} (1-\delta_{2^r})	(1-\delta_{2^{r-1}})
	\geq  c_4, \]
	for a constant $c_4=c_4(c_3)>0$.
	Hence \eqref{eq:dim.reduction.proof} can be applied iteratively starting with $r=s$ and until $r=1$---thus completing the proof.
\end{proof}

In order to derive probabilistic estimates on the terms that appear in Lemma \ref{lem:dim.reduction}, the following (standard) observation is required:
\begin{lemma}
\label{lem:monotone.rearragment.decay}
	Let $c_1$ and $c_2$ be constants.
	Then there are constants $R=R(c_1,c_2,\beta)$ and $c_3=c_3(c_1,c_2,\alpha,\beta, L)$  such that the following holds.
	Let $1\leq 2^s \leq n$ and let $\mathcal{Z}$ be a collection of random variables satisfying
	\begin{align*}
	|\mathcal{Z}|
	&\leq \exp\Big( c_1 2^s \log\frac{en}{2^s} \Big),\\
	\| Z \|_{L_p}
	&\leq  L p^{1/\alpha} \| Z \|_{L_2}
	\quad \text{for every } 2\leq p\leq  R \cdot \log en \text{ and every }Z\in\mathcal{Z}.
	\end{align*}
	Then, for every $u\geq e$, with probability at least
	\[1-\exp\Big( - c_2\beta  (\log u ) \cdot  2^s \log\frac{en}{2^s}\Big),\]
	\[ \max_{Z\in\mathcal{Z}} Z^\ast_{2^s} \leq c_3   u  \log^{1/\alpha}\Big( \frac{en}{2^s}\Big) \max_{Z\in\mathcal{Z}} \| Z \|_{L^2};\]
	here $(Z_i)_{i=1}^n$ denote independent copies of each random variable $Z\in\mathcal{Z}$.
\end{lemma}

\begin{proof}
	Set $U:=\max_{Z\in\mathcal{Z}} \| Z \|_{L_2}$ and let $v\geq 0$.
	Using the union bound, a binomial estimate, and the independence  $(Z_i)_{i=1}^n$ for each $Z$, it follows that
	\begin{align}
	\nonumber
	&\P\Big(\max_{Z\in\mathcal{Z}} Z^\ast_{2^s} \geq v U \Big) \\
	\nonumber
	&\leq \sum_{Z\in\mathcal{Z}} \P\Big( \text{there is } I\subset\{1,\dots,n\} \text{ s.t.\ } |I|={2^s} \text{ and } |Z_i|\geq v  U \text{ for all  } i\in I \Big) \\
	\nonumber	
	&\leq  \sum_{Z\in\mathcal{Z}} \binom{n}{{2^s}} \P^{2^s} \big( |Z|\geq v  U \big) \\
	\label{eq:decreasing.order.estimate}	
	&\leq \sum_{Z\in\mathcal{Z}} \exp\Big( {2^s} \log\frac{en}{{2^s}} + 2^s \log \P( |Z|\geq v  U )\Big).
	\end{align}
	
	Set
	\[R:=c_1+c_2\beta+1
	\quad\text{and}\quad
	p:=R\log \frac{en}{{2^s}},\]
	and assume without loss of generality  that $p\geq 2$.
	Then, for $v:=L p^{1/\alpha} u$, Markov's inequality and the assumption on the growth of moments imply that
	\[ \max_{Z\in\mathcal{Z}} \P\big( |Z|\geq v U \big)
	\leq \Big(\frac{ p^{1/\alpha} U }{ v U }\Big)^p
	= u^{-p}.\]
The claim follows from \eqref{eq:decreasing.order.estimate} because $v\leq c_3(c_1,c_2,\alpha,\beta, L) u \log^{1/\alpha}(\frac{en}{{2^s}})$.
\end{proof}

\begin{lemma}
\label{lem:estimate.alpha.2s}
	There is a constant $c_1=c_1(\alpha, \beta,L)$ such that, for every $1\leq 2^s\leq n$, with probability $\P_{\mathbb{X}}$ at least $1-n^{-2\beta}$,
	\[  \OD_{2^s} \leq c_1 \sqrt{ \frac{2^s}{m} } \log^{1/\alpha}\Big( \frac{en}{2^s} \Big) \ODN_{2^s}.  \]
\end{lemma}
\begin{proof}
Following the notation of Lemma \ref{lem:dim.reduction}, one may integrate  with respect to $\eta$ Inequality \eqref{eq:alpha.dimension.reduction}, obtained in Lemma \ref{lem:dim.reduction}. Hence,
	\[ \OD_{2^s}
	\leq c_2\Big( \OD_{2^0} + \sum_{r=1}^s \big( \E_\eta \mathcal{E}_{I_\eta,2^r} + \E_\eta \mathcal{F}_{I^c_\eta,2^r} \big) \Big).\]
We shall only estimate the terms $\E_\eta \mathcal{E}_{I_\eta,2^r}$; the required bounds on $\OD_{2^0}$ and on $\E_\eta \mathcal{F}_{I^c_\eta,2^r}$ are based on the same argument and are omitted.

Let us show that, with $\P_{\mathbb{X}}$-probability at least $1-n^{-2\beta}$,
	\[ \sum_{r=1}^s \E_\eta  \mathcal{E}_{I_\eta,2^r}
	\leq c_1(\alpha,\beta,L) \sqrt{ \frac{2^s}{ m } } \log^{1/\alpha}\Big( \frac{en}{2^s} \Big) \ODN_{2^s} .\]
	
	{\it Step 1.} Fix $I\subset\{1,\dots,n\}$.
	Then, for $1\leq 2^r\leq n$,
	\[ \mathcal{E}_{I,2^r}
	\leq  \max_{y\in \mathcal{S}'_{I^c,2^{r}}} \sqrt \frac{  2^r }{ m }  \Big\langle X_i, \frac{1}{\sqrt m } \sum_{j\in I^c} y_jX_j\Big\rangle^{\ast}_{2^{r-1}}.\]
Indeed, this holds by replacing all the terms in $\{2^{r-1}+1,\dots,2^r\}$ by the largest one.

Consider the set 	
\[
\mathcal{Z}:=\Big\{  \Big\langle X, \frac{1}{\sqrt m } \sum_{j\in I^c} y_jX_j\Big\rangle : y\in \mathcal{S}'_{I^c,2^{r}} \Big\}.
\]
Note that $\mathcal{Z}$ consist of at most $|\mathcal{S}'_{I^c,2^{r}}|$ elements, i.e., it is of cardinality at most $\exp(c_3 2^r\log\frac{en}{2^r})$ (see Lemma \ref{lem:dim.reduction}). Moreover, for every $Z \in \mathcal{Z}$ and every $2\leq p \leq R\log en$,
	\[
\| Z \|_{L_p}
	\leq L p^\frac{1}{\alpha} \| Z \|_{L_2}
	\leq L p^\frac{1}{\alpha} \ODN_{I^c,2^r}.
\]

Applying Lemma \ref{lem:monotone.rearragment.decay} (conditionally on $(X_i)_{i \in I^c}$), and for an absolute constant $c_4$ to be chosen later, there are constants $R=R(\beta,c_3,c_4)$ and $c_5=c_5(c_3,c_4,\alpha,\beta,L)$ such that
	\begin{align}
	\label{eq:alpha.estimate}
	\begin{split}
	&\P_{X_{I}}\Big( \mathcal{E}_{I,2^r} \geq  c_5 u \sqrt{ \frac{2^r}{ m } } \log^{1/\alpha}\Big(\frac{en}{2^r} \Big)\ODN_{I^c,2^r}  \Big) \\
	&\leq \exp\Big(-c_4 \beta (\log u ) 2^r \log\frac{en}{2^r} \Big)
	\end{split}
	\end{align}
	for every $u\geq e$.
	
	Since $\ODN_{I^c,2^r}\leq \ODN_{2^r}$ for every realization of $(X_i)_{i=1}^n$, it follows from Fubini's Theorem that one may replace  $\P_{X_I}$ and $\ODN_{I^c,2^r}$ in \eqref{eq:alpha.estimate} by $\P_{\mathbb{X}}$ and $\ODN_{2^r}$, respectively.
	In particular, by the union bound, with $\P_{\mathbb{X}}$-probability at least
	\begin{align*}
	1-\sum_{r=1}^s \exp\Big( - c_4 \beta (\log u)  2^r \log\frac{en}{2^r} \Big)
	&\geq 1- \exp\big( -c_4c_6 \beta (\log u)  \log n \big) \\
	&= 1- u^{-c_4c_6\beta\log n} ,
	\end{align*}
we have that
	\begin{align*}
	\sum_{r=1}^s \mathcal{E}_{I,2^r}
	&\leq c_5 \sum_{r=1}^s u \sqrt{ \frac{2^r}{ m } } \log^{1/\alpha}\Big( \frac{en}{2^r} \Big) \ODN_{2^r}  \\
	&\leq c_7 c_5 u \sqrt{ \frac{2^s}{ m } }  \log^{1/\alpha}\Big( \frac{en}{2^s} \Big) \ODN_{2^s}
	= :u \cdot (\ast).
	\end{align*}
	Let the constant $c_4$ be large enough to ensure that $c_4c_6\geq 4$.
	Thus, using tail integration,
	\begin{align}
	\label{eq:EX.beta}
	\Big(\E_{\mathbb{X}} \Big| \frac{\sum_{r=1}^s \mathcal{E}_{I,2^r} }{(\ast)} \Big|^p\Big)^{1/p}
	\leq c_8
	\end{align}
for $p:=2\beta\log n$.

	{\it Step 2.}
 	As  \eqref{eq:EX.beta} holds for any $I\subset\{1,\dots,n\}$, one can integrate over $\eta$ and apply Jensen's inequality (just as in the proof of Corollary \ref{cor:V.probability.estimate}). Hence,
	\[ \Big( \E_{\mathbb{X}} \Big| \frac{ \sum_{r=1}^s \E_\eta \mathcal{E}_{I_\eta,2^r} }{(\ast)} \Big|^p\Big)^{1/p}
	\leq c_8, \]
	and by  Markov's inequality,
	\[ \P_{\mathbb{X}}\Big( \sum_{r=1}^s  \E_\eta \mathcal{E}_{I_\eta,2^r}  \geq c_8 e \cdot   (\ast) \Big)
	\leq \exp(-p)
	= n^{-2\beta}, \]
	as required.
\end{proof}	

\begin{proof}[Proof of Theorem \ref{thm:random.matrix.satisfies.assumption}]
	The proof relies on a self-bounding argument.
	In a first step, following the decoupling argument detailed in Section \ref{sec:decoupling}, for every $1\leq 2^s\leq n$, we obtain
	\begin{align*}
	\nonumber
	(\ODN_{2^s})^2
	&=\max_{\{x \in S^{n-1} : |{\rm supp}(x)| \leq 2^s\}} \Big( \sum_{i=1}^n \frac{ \|X_i\|_2^2 }{ m } x_i^2 + \sum_{i,j=1, \, i\neq j}^m x_ix_j \frac{ \langle X_i,X_j\rangle }{ m } \Big)\\
	\nonumber
	&\leq \max_{1 \leq i \leq n}\frac{\|X_i\|_2^2}{ m } + 4 \OD_{2^s}.
	\end{align*}
	The assumption that $X$ is suitable implies that, with $\P_{\mathbb{X} }$-probability at least $1-\gamma$, for every $1\leq 2^s\leq n$,
	\begin{align}
	\label{eq:alpha.1}
	(\ODN_{2^s})^2
	&\leq 2 + 4\OD_{2^s}.
	\end{align}
	
	Next, observe that
	\[ |\{ s: 1\leq  2^s\leq n\}|\leq \log n
	\quad\text{and}\quad
		(\log n ) n^{-2\beta}\leq n^{-\beta}.\]
	Combining Lemma \ref{lem:estimate.alpha.2s} and the union bound, there is a constant $c_2=c_2(\alpha,\beta,L)$, such that, with $\P_{\mathbb{X}}$-probability at least $1-n^{-\beta}$,
	\begin{align}
	\label{eq:alpha}
	\OD_{2^s}
	\leq c_2 \sqrt{ \frac{2^s}{m} } \log^{1/\alpha}\Big( \frac{en}{2^s} \Big) \ODN_{2^s}
	\end{align}
	for all  $1\leq 2^s\leq n$.

	Thus, on the intersection of the events on which \eqref{eq:alpha.1} and \eqref{eq:alpha} hold (which has $\P_{\mathbb{X}}$-probability at least $1-\gamma-n^{-\beta}$) we conclude that
	\[ \ODN_{2^s}
	\leq  \max\Big\{ \sqrt{2} , c_2 \sqrt{ \frac{ 2^s }{ m } } \log^{1/\alpha}\Big( \frac{en}{2^s} \Big) \Big\}\]
	for all $1\leq 2^s\leq n$.
	Using this estimate for $\ODN_{2^s}$ in \eqref{eq:alpha} completes the proof.	
\end{proof}

\bibliographystyle{plain}

\vspace{1em}
\noindent
\textsc{Acknowledgements:}
Daniel Bartl is grateful for financial support through the Vienna Science and Technology Fund (WWTF) project MA16-021 and the Austrian Science Fund (FWF) project ESP-31N and project P28661.

\end{document}